\documentclass[a4paper,reqno]{amsart}

\usepackage{amsmath, amssymb, amsthm}
\usepackage[english]{babel}
\usepackage{enumerate}
\usepackage[all]{xy}
\usepackage{graphicx}
\usepackage{mathtools}
\usepackage{mathrsfs}
\usepackage{hyperref}
\usepackage{verbatim}
\usepackage{color}
\usepackage{tikz-cd}
\tikzcdset{arrow style=tikz, diagrams={>=stealth}}

\allowdisplaybreaks
\newcommand*{\id}{\textup{id}}

\numberwithin{equation}{section}

\theoremstyle{plain}

\newtheorem{thm}{Theorem}[section]

\newtheorem{defi}[thm]{Definition}

\theoremstyle{remark}

\newtheorem{rem}[thm]{Remark}

\numberwithin{equation}{section}

\newcommand{\ot}{\otimes}

\newcommand{\beq}{\begin{equation}}
\newcommand{\eeq}{\end{equation}}

\newcommand{\cR}{\mathcal{R}}

\newcommand{\rlbicross}{{\triangleright\!\!\!\blacktriangleleft}}

\newcommand{\lbiprod}{{>\!\!\!\triangleleft\kern-.33em\cdot}}
\newcommand{\rbiprod}{{\cdot\kern-.33em\triangleright\!\!\!<}}

\newcommand{\uz}{{}^{\scriptscriptstyle{[0]}}}
\renewcommand{\o}{{}_{\scriptscriptstyle{(1)}}}

\newcommand{\umo}{{}^{\scriptscriptstyle{[-1]}}}
\renewcommand{\t}{{}_{\scriptscriptstyle{(2)}}}

\renewcommand{\th}{{}_{\scriptscriptstyle{(3)}}}

\newcommand{\fo}{{}_{\scriptscriptstyle{(4)}}}

\newcommand{\five}{{}_{\scriptscriptstyle{(5)}}}
\newcommand{\si}{{}_{\scriptscriptstyle{(6)}}}
\newcommand{\se}{{}_{\scriptscriptstyle{(7)}}}

\newcommand{\ho}{{}^{\scriptscriptstyle{(1)}}}
\newcommand{\htw}{{}^{\scriptscriptstyle{(2)}}}

\newcommand{\la}{{\triangleright}}
\newcommand{\ra}{{\triangleleft}}

\DeclareMathOperator{\tens}{\otimes}

\begin{document}

\author{Xiao Han}
\address{Queen Mary University of London\\
		School of Mathematical Sciences, Mile End Rd, London E1 4NS, UK}

\email{x.h.han@qmul.ac.uk}


\title{On bicrossed modules of Hopf algebras}
%

\begin{abstract}
We use Hopf algebroids to formulate a notion of a noncommutative and non-cocommutative Hopf 2-algebra. We show how these arise
from a bicrossproduct Hopf algebra with Peiffer identities. In particular, we show that for a Hopf algebra $H$ with bijective antipode, the mirror bicrossproduct Hopf algebra $H\rlbicross H_{cop}$ is a Hopf 2-algebra. We apply the theorem on Sweedler's Hopf algebra as an example.
 \end{abstract}

\maketitle
\begin{quote}  2020 Mathematics Subject Classification: 16T05,  20G42\end{quote}
\section{Introduction}

It is given in \cite{BL} that a strict 2-group is a strict 2-category with one object in which all 1-morphisms and 2-morphisms
are invertible. It is also known \cite{Por08} that a strict 2-group is equivalent to a crossed module, which consists of two groups equipped with the so-called  Peiffer identity. Motivated by the highly developed theory of quantum groups and quantum groupoids with significant application in math and physics,
we are going to use these to
formulate a Hopf 2-algebra as a quantization of a 2-group.

In \cite{Yael11}, Y. Fr\'egier and  F. Wagemann had constructed a `quantum' crossed module in terms of a crossed module and a crossed comodule of Hopf algebras. In \cite{XH}, we had constructed a Hopf 2-algebra in the sense of a quantization of a strict 2-group. However, the base algebra of the
 Hopf 2-algebra is commutative.

In this paper, we manage to construct noncommutative and non-cocommutative Hopf 2-algebras (with possibly noncommutative base algebras) in terms of bicrossproduct Hopf algebras with Peiffer identities. On the one hand, as a quantization of a strict 2-group, a Hopf 2-algebra consists of a  Hopf algebra structure on the base algebra $B$,
which plays the role of the quantization of the `1-morphism' group. A Hopf 2-algebra also consists of a Hopf algebra and a Hopf algebroid structure on the same underlying algebra $H$, which play the role of the quantization of the `horizontal' group and `vertical' groupoid structure on the `2-morphism', such that the `horizontal'  and `vertical' coproducts are cocommutative, i.e.
\[(\blacktriangle\otimes \blacktriangle)\circ \Delta=(\id_{H}\otimes \mathrm{flip}\otimes \id_{H})\circ (\Delta\otimes_{B}\Delta)\circ \blacktriangle,\]
where $\Delta$ is the coproduct of the Hopf algebra $H$ and $\blacktriangle$ is the coproduct of the Hopf algebroid $H$.
 On the other hand, in the view of quantization of a crossed module, we have combined both crossed module and crossed comodule structures on a bicrossproduct Hopf algebra in Theorem \ref{thm. bicrossed module}, and we call the resulting structure a bicrossed module of Hopf algebras. The main result of this paper is to show that a bicrossed module of Hopf algebras is a Hopf 2-algebra. Moreover, for any Hopf algebra $H$ with a bijective antipode, we can show that the `mirror' bicrossproduct Hopf algebra \cite{Majid1} with the inverse of the antipode $(H\rlbicross H_{cop}, S^{-1}_{H})$ is a bicrossed module of Hopf algebras and hence a Hopf 2-algebra. Finally, we give a concrete example $(H^{4}\rlbicross H^{4}_{cop}, S^{-1}_{H^{4}})$, where $H^{4}$ is the Sweedler's Hopf algebra, and this results in  a noncommutative and non-cocommutative Hopf 2-algebra since $H^{4}$ is neither commutative nor cocommutative.

We also note another interesting work \cite{Majid3}, which constructs a  quantum 2-group in a different approach.

\subsection*{Acknowledgements} I thank Shahn Majid, Ralf Meyer and Chenchang Zhu for helpful discussions. This project has received funding from the European Union’s Horizon 2020 research and innovation programme under the Marie Sklodowska-Curie grant agreement No 101027463”.

\section{Basic algebraic preliminaries} \label{sec2}
A bialgebra $H$ is an unital algebra and a counital coalgebra, such that the coproduct $\Delta$ and the counit $\varepsilon$ are algebra maps. If we use the sumless Sweedler notation $\Delta(h)=h\o\otimes h\t$ then $(hg)\o\ot (hg)\t=h\o g\o\ot h\t g\t$  and $\varepsilon(hg)=\varepsilon(h)\varepsilon(g)$ for any $h, g\in H$. A  Hopf algebra is a bialgebra with a linear map $S: H\to H$, such that $h\o S(h\t)=\varepsilon(h)=S(h\o) S(h\t)$. The opposite bialgebra (resp. coopposite bialgebra) of $H$ is denoted by $H^{op}$ (resp. $H_{cop}$). A right module algebra of an Hopf algebra $H$ is an algebra $A$ and a right $H$-module with a right action $\ra$ of $H$, such that $(ab)\ra h=(a\ra h\o)(b\ra h\t)$ and $1_{A}\ra h=\varepsilon(h)$, for any $a,b\in A$ and $h\in H$. Dually, a left comodule algebra of $H$ is a coalgebra and a left $H$-comodule with a left coaction $\delta: C\to H\ot C$, such that $c\umo\ot c\uz\o\ot c\uz\t=c\o\umo c\t\umo\ot c\o\uz \ot c\t\uz$ and
$c\umo \varepsilon(c\uz)=\varepsilon (c)$, for any $c\in C$. Here we use the extended Sweedler notation $\delta(c)=c\umo\ot c\uz$.

Let $B$ be an unital algebra over the field $k$, a $B$-coring $C$ is a coalgebra in ${}_{B}\mathcal{M}_{B}$. In other words, there is a coproduct $\Delta: C\to C\ot_{B}C$ and an counit $\delta: C\to B$, such that $(\Delta\ot_{B}\id)\circ \Delta=(\id\ot_{B}\Delta)\circ \Delta$ and $(\varepsilon\ot_{B}\id)\circ \Delta=(\id\ot_{B}\varepsilon)\circ \Delta=\id$.
We also recall the definition right bialgebroids:
\begin{defi} \cite{BS} \label{def:right.bgd} Let $B$ be a unital algebra. A right $B$-bialgebroid is an algebra $\cR$ with commuting algebra maps (`source' and `target' maps) $s:B\to \cR$ and $t:B^{op}\to \cR$ and a $B$-coring for the bimodule structure $b\la X\ra b'=X s(b') t(b)$ for any $X\in \cR$ and $b, b'\in B$, which is compatible in the sense
\begin{itemize}
\item[(i)] The coproduct $\Delta$ corestricts to an algebra map  $\cR\to \cR\times_{B} \cR$ where
\begin{equation*} \cR\times_{B} \cR :=\{\ \sum_i X_i \ot_{B} Y_i\ |\ \sum_i s(b)X_i \ot_{B} Y_i=
\sum_i X_i \ot_{B}  t(b)Y_i  ,\quad \forall b\in B\ \}\subseteq \cR\tens_{B}\cR,
\end{equation*}
 is an algebra via factorwise multiplication.
\item[(ii)] The counit $\varepsilon$ is a right character in the following sense:
\begin{equation*}\varepsilon(1_{\cR})=1_{B},\quad \varepsilon(s(\varepsilon(X))Y)=\varepsilon(XY)=\varepsilon(t(\varepsilon(X))Y)\end{equation*}
for all $X,Y\in \cR$ and $b\in B$.
\end{itemize}
\end{defi}
In the following, we will also use the sumless Sweedler indexes to denote the coproducts of right bialgebroids, namely, $\Delta(X)=X\ho\ot_{B}X\htw$.

\begin{defi}\label{defHopf1}
A right bialgebroid $\cR$ is a right Hopf algebroid if
\[\lambda: \cR\ot_{B^{op}}\cR\to \cR\ot_{B}\cR,\quad
    \lambda(X\ot_{B^{op}} Y)=XY\ho\ot_{B}Y\htw\]
is invertible, where $\tens_{B^{op}}$ is induced by $t$ (so $Xt(b)\ot_{B^{op}}Y=X\ot_{B^{op}}t(b)Y$, for all $X, Y\in \cR$ and $b\in B$) while the $\tens_{B}$ is the standard one, i.e. $Xs(b)\ot_{B}Y=X\ot_{B}Yt(b)$.\\
Similarly, A right bialgebroid $\cR$ is an anti-right Hopf algebroid if
\[\mu: \cR\ot^{B^{op}}\cR\to \cR\ot_{B}\cR,\quad
    \mu(X\ot^{B^{op}} Y)=X\ho\ot_{B}YX\htw\]
is invertible, where $\tens^{B^{op}}$ is induced by $s$ (so $s(b)X\ot^{B^{op}}Y=X\ot^{B^{op}}Ys(b)$, for all $X, Y\in \cR$ and $b\in B$) while the $\tens_{B}$ is the standard one, i.e. $Xs(b)\ot_{B}Y=X\ot_{B}Yt(b)$.
\end{defi}
We adopt the shorthand
\begin{equation}\label{X+-1} X_{-}\ot_{B^{op}} X_{+} :=\lambda^{-1}(1\ot_{B}X),\quad X_{[+]}\ot^{B^{op}}X_{[-]}:=\mu^{-1}(X\ot_{B}1).\end{equation}
Moreover, as a stronger version of (anti-)right Hopf algebroids, there is a definition of full Hopf algebroids in terms of right bialgebroids, which is equivalent to the left-handed version given in \cite{BS}.
\begin{defi}\label{def. full Hopf algebroid3}\cite{XH1}
    A right bialgebroid $\cR$ over $B$ is a full Hopf algebroid, if there is an invertible anti-algebra map $S:\cR\to \cR$, such that
    \begin{itemize}
        \item $S\circ t=s$,
        \item $X\htw S^{-1}(X\ho)\ho\ot_{B}S^{-1}(X\ho)\htw=1\ot_{B}S^{-1}(X)$,
        \item $S(X\htw)\ho\ot_{B}X\ho S(X\htw)\htw=S(X)\ot_{B}1$.
    \end{itemize}
\end{defi}

\section{Hopf-2-algebras and bicrossed modules of Hopf algebras}\label{sec. Hopf 2-algebras and bicrossed modules of Hopf algebras}

Recall that, a strict 2-group consists of a `1-morphism' group $\mathcal{G}^{0}$ and a `2-morphism' group $\mathcal{G}^{1}$ with `horizontal' group multiplication $\circ: \mathcal{G}^{1}\times \mathcal{G}^{1}\to \mathcal{G}^{1}$. Moreover, $\mathcal{G}^{1}$ is a groupoid with `vertical' groupoid multiplication $\bullet: \mathcal{G}^{1}{}_{s}\times_{t} \mathcal{G}^{1}\to \mathcal{G}^{1}$, such that the `horizontal' and `vertical' products are commutative in the sense $(g\bullet h)\circ (k\bullet l)=(g\circ k)\bullet (h\circ l)$, for any $(g, h), (k, l)\in \mathcal{G}^{1}{}_{s}\times_{t} \mathcal{G}^{1}$. With such a classical picture in mind, we can construct a `quantum' version of a 2-group (in a similar way as in \cite{XH} but the base algebra is not necessarily commutative), which consists of a pair of Hopf algebras $B$ and $H$, such that $H$ is also a Hopf algebroid over $B$ with a cocommutative relation between the Hopf algebra coproduct and the Hopf algebroid coproduct.

\begin{defi}\label{def. Hopf-2-algebra}
A {\em (anti-)right Hopf 2-algebra} consists of a pair of Hopf algebras
$(B, m_{B}, 1_{B}, \Delta_{B}, \varepsilon_{B}, S_{B})$, $(H, m, 1_{H}, \Delta, \varepsilon, S)$, and a (anti-)right Hopf algebroid $(H, m, 1_{H}, \blacktriangle, \varepsilon_{H}, s, t)$ over $B$, such that
\begin{itemize}
    \item [(i)] The underlying algebra structures of the Hopf algebra $(H, m, 1_{H}, \Delta, \varepsilon, S)$ and the (anti-)right Hopf algebroid $(H, m, 1_{H}, \blacktriangle, \varepsilon_{H}, s, t)$ are the same.
    \item[(ii)] The counit $\varepsilon_{H}: H\to B$ is a coalgebra map.
    \item[(iii)] The source and target maps are morphisms of bialgebras.
    \item[(iv)] The two coproducts $\Delta$ and $\blacktriangle$ have the following cocommutation relation:
    \begin{align}\label{equ. cocommutative relation}
        (\blacktriangle\otimes \blacktriangle)\circ \Delta=(\id_{H}\otimes \mathrm{flip}\otimes \id_{H})\circ (\Delta\otimes_{B}\Delta)\circ \blacktriangle,
    \end{align}
    where $\id_{H}\ot \mathrm{flip} \ot \id_{H}:H\ot H\ot_{B\ot B}H\ot H\to H\ot_{B} H\ot H\ot_{B} H$ is given by  $\id_{H}\ot \mathrm{flip} \ot \id_{H}: X\ot  Y\ot_{B\ot B} Z\ot  W\mapsto (X\ot_{B}  Z)\ot (Y\ot_{B}W)$.
\end{itemize}
A (anti-)right Hopf 2-algebra is a \textrm{Hopf 2-algebra}, if the  right bialgebroid $(H, m, 1_{H}, \blacktriangle, \varepsilon_{H}, s, t)$ is a full Hopf algebroid.
\end{defi}

An (anti-)left Hopf 2-algebra can be similarly defined based on a left bialgebroid. Recall that we use different
sumless Sweedler notations for the coproducts of Hopf algebra and Hopf algebroid, namely, $\Delta(X)=X\o \otimes X\t$, $\blacktriangle(X)={}X\ho \ot_{B} X\htw$, for any $X\in H$.
We can see condition  (iv) is well defined since $H\ot H$ has the $B\ot B$-bimodule structure, which is given by
    $(b\ot b')\triangleright(X\ot Y)=Xt(b)\ot Yt(b')$ and $(X\ot  Y)\triangleleft (b\ot b')=Xs(b)\ot Ys(b')$, for any $b\ot b'\in B\ot B$ and $X\ot  Y\in H\ot H$. Indeed, for any $b\in B$ and $X$, $Y\in H$ we have
    \begin{align*}
        (\Delta\ot_{B} \Delta)(X\ot_{B} b\triangleright  Y)=& (\Delta\ot_{B} \Delta)(X\ot_{B}Yt(b))\\
        =&({}X\o\ot {}X\t)\ot_{B\ot B}(Y\o t(b)\o\ot Y\t t(b)\t)\\
        =&({}X\o\ot {}X\t)\ot_{B\ot B}(Y\o t(b\o)\ot Y\t t(b\t))\\
        =&(X\o s(b\o)\ot X\t s(b\t))\ot_{B\ot B}({}Y\o\ot{}Y\t)\\
        =&(X\o s(b)\o\ot X\t s(b)\t)\ot_{B\ot B}({}Y\o\ot{}Y\t)\\
        =&(\Delta\ot_{B} \Delta)(X\triangleleft b \ot_{B} Y),
    \end{align*}
    where the 2nd step uses the fact that $\Delta$ is an algebra map, and the 3rd step uses the fact that $t$ is a coalgebra map. Clearly, the map $\id_{H}\ot \mathrm{flip} \ot \id_{H}: X\ot  Y\ot_{B\ot B} Z\ot  W\mapsto (X\ot_{B}  Z)\ot (Y\ot_{B}W)$ is also well defined for any $X\ot Y, Z\ot  W\in H\ot H$. More precisely, (iv) can be written as
    \begin{align}\label{coco}
{}X\o\ho\ot_{B}X\o{}\htw\ot X\t\ho \ot_{B}X\t\htw={}X\ho\o\ot_{B}X\htw\o\ot X\ho\t\ot_{B}X\htw\t,
    \end{align}
    for any $X\in H$.

To observe that a Hopf 2-algebra is a quantum version of a 2-group,  we first give an example.
\subsection{Example}
Let $(\mathcal{G}^{0},  \mathcal{G}^{1})$ be a finite 2-group, it is given by \cite{AryGho2} that $(A(\mathcal{G}^{0}),  A(\mathcal{G}^{1}))$ the algebras of functions on  $(\mathcal{G}^{0},  \mathcal{G}^{1})$ is a Hopf algebroid. As the algebra $A(\mathcal{G}^{0})$ (resp. $A(\mathcal{G}^{1}$)) is spanned by
elements of the form $f_{g}$ for all $g\in \mathcal{G}^{0}$ (resp. $f_{e}$ for all $e\in \mathcal{G}^{1}$). The multiplication for $(A(\mathcal{G}^{0}),  A(\mathcal{G}^{1}))$ is determined by
\[f_{e}\, f_{e'}=\delta_{e,e'}f_{e},\quad\forall e,e'\in \mathcal{G}^{1};\qquad f_{g}\, f_{g'}=\delta_{g,g'}f_{g},\quad\forall g,g'\in \mathcal{G}^{0}.\]
The units are
\[1_{A(\mathcal{G}^{1})}=\sum_{e\in \mathcal{G}^{1}}f_{e},\qquad 1_{A(\mathcal{G}^{0})}=\sum_{g\in \mathcal{G}^{0}}f_{g}.\]
As  a Hopf algebroid, the source and target maps are
\[s(f_{g})=\sum_{e:s(e)=g}f_{e},
\qquad t(f_{g})=\sum_{e:t(e)=g}f_{e}.\]
The coring structure is
\[\varepsilon_{H}(f_{e})=\sum_{g\in \mathcal{G}^{0}}\delta_{e,\id_{g}}f_{g},\qquad \blacktriangle(f_{e})=\sum_{e_{1},e_{2}:e_{2}\bullet e_{1}=e}f_{e_{1}}\ot_{A(\mathcal{G}^{0})}f_{e_{2}},\]
where $\bullet: \mathcal{G}^{1}{}_{s}\times_{t} \mathcal{G}^{1}\to \mathcal{G}^{1}$ is the `vertical' groupoid multiplication.
The Hopf algebroid antipode is
\[S(f_{e})=f_{e^{-1}},\]
Moreover, since both $A(\mathcal{G}^{0})$ and $A(\mathcal{G}^{1})$ are finite groups, they are also Hopf algebras. More explicitly, the coproducts are
\[\Delta(f_{e})=\sum_{e_{1},e_{2}:e_{1}\circ e_{2}=e}f_{e_1}\ot f_{e_2};\qquad \Delta(f_{g})=\sum_{g_{1},g_{2}:g_{1}\circ g_{2}=g}f_{g_1}\ot f_{g_2},\]
where $\circ: \mathcal{G}^{1}\times \mathcal{G}^{1}\to \mathcal{G}^{1}$ (resp. $\circ: \mathcal{G}^{0}\times  \mathcal{G}^{0}\to \mathcal{G}^{0}$) are the `horizontal' group multiplications. As all the above structures are the dualization of the structure of a 2-group, it is not hard to check $\varepsilon_{H}$, $s$, and $t$ are all coalgebra maps. For (iv) of Definition \ref{def. Hopf-2-algebra}, we have on the one hand,
\begin{align*}
    (\blacktriangle\otimes \blacktriangle)\circ \Delta(f_{e})=\sum_{e_{1},e_{2},e_{3},e_{4}:(e_{2}\bullet e_{1})\circ (e_{4}\bullet e_{3})=e}f_{e_{1}}\ot_{B}f_{e_{2}}\ot f_{e_{3}}\ot_{B}f_{e_{4}}.
\end{align*}
On the other hand,
\begin{align*}
    (\id_{H}\otimes \mathrm{flip}\otimes \id_{H})\circ (\Delta\otimes_{B}\Delta)\circ \blacktriangle(f_{e})=\sum_{e_{1},e_{2},e_{3},e_{4}:(e_{2}\circ e_{4})\bullet (e_{1}\circ e_{3})=e}f_{e_{1}}\ot_{B}f_{e_{2}}\ot f_{e_{3}}\ot_{B}f_{e_{4}}.
\end{align*}
They are equal since the vertical and horizontal products commute.
\begin{rem}
We know that for a classical group $G$, its linearization $\mathbb{K}G$ is Hopf algebra. More precisely, its coalgebra structure is given by
\[\Delta(g)=g\ot g,\qquad \varepsilon(g)=1.\]
However, given a 2-group $(\mathcal{G}^{0}, \mathcal{G}^{1})$, its (partial) linearization is in general not a Hopf 2-algebra.

Scenario I. $H=\mathbb{K}\mathcal{G}^{1}$, $B=A(\mathcal{G}^{0})$. By \cite{AryGho2}, $H$ is a $B$-Hopf algebroid with
\[s(f_{g})=t(f_{g})=\id_{g},\quad \blacktriangle(e)=e\ot_{B}e,\quad \varepsilon_{H}(e)=f_{t(e)},\quad S(e)=e^{-1},\quad 1=\sum_{g\in \mathcal{G}^{0}}\id_{g},\]
for any $g\in \mathcal{G}^{0}$ and $e\in \mathcal{G}^{1}$. However, in general $\varepsilon_{H}$ is not a coalgebra map. Indeed, on the one hand,
\begin{align*}
    (\varepsilon_{H}\ot \varepsilon_{H})\circ \Delta(e)=\varepsilon_{H}(e)\ot \varepsilon_{H}(e)=f_{t(e)}\ot f_{t(e)}.
\end{align*}
On the other hand,
\begin{align*}
    \Delta_{B}\circ \varepsilon_{H}(e)=\sum_{g_1,g_{2}:g_{1}\circ g_{2}=t(e)}f_{g_{1}}\ot f_{g_{2}}.
\end{align*}

Scenario II. $H=\mathbb{K}\mathcal{G}^{1}$, $B=\mathbb{K}\mathcal{G}^{0}$. In this case, $H$ is in general not even a $B$-Hopf algebroid. The linearization of both $\mathcal{G}^{1}$ and $\mathcal{G}^{0}$ is studied in \cite{Majid3} where $(\mathbb{K}\mathcal{G}^{1}, \mathbb{K}\mathcal{G}^{0})$ is called embedded
quantum groupoid.
\end{rem}

\subsection{Bicrossed modules of Hopf algebras}
Motivated by the equivalence between strict 2-groups and crossed modules. Here we first recall the definition of bicrossproduct Hopf algebras and then construct a bicrossed module over a bicrossproduct Hopf algebra by combining the so-called Peiffer identity and co-Peiffer identity \cite{Yael11} in a compatible way.

\begin{defi}(cf. \cite{Majid1} \cite{Majid2})
    Let  $A$ and $B$ be Hopf algebras, such that
    \begin{itemize}
        \item [(i)] $A$ is a left $B$-comodule coalgebra with coaction $\delta$, $B$ is a right $A$-module algebra with action $\ra$,
        \item[(ii)] $\varepsilon(b\ra a)=\varepsilon(b)\varepsilon(a)$ and $\delta(1)=1\ot 1$,
        \item[(iii)] $(b\ra a)\o\ot (b\ra a)\t=(b\o\ra a\o)a\t\umo\ot  b\t\ra a\t\uz$,
        \item[(iv)] $a\o\umo(b\ra a\t)\ot a\o\uz=(b\ra a\o)a\t\umo\ot a\t\uz$.
    \end{itemize}
    The \textit{bicrossproduct} of $A$ and $B$ is a Hopf algebra $A \rlbicross B$ whose underlying vector space is $A\ot B$, with the coproduct and counit
    \[\Delta(a\ot b)=a\o\ot a\t\umo b\o\ot a\t\uz\ot b\t, \quad \varepsilon(a\ot b)=\varepsilon(a)\varepsilon(b),\]
    the product and unit
    \[(a\ot b)(a'\ot b')=aa'\o\ot (b\ra a'\t) b', \quad 1\ot 1.\]
    Moreover, the antipode is
    \[S(a\ot b)=S(a\uz\t)\ot S(a\umo b)\ra S(a\uz\o),\]
    for any $a, a'\in A$ and $b, b'\in  B$.
\end{defi}

\begin{thm}\label{thm. bicrossed module}
        Let $A\rlbicross B$ be a bicrossproduct of two Hopf algebras $A$ and $B$, if the antipode of $A$ is bijective and $\phi: B^{op}\to A$ is a bialgebra morphism, such that    \begin{itemize}
        \item [(1)] $\phi(b)\umo\ot\phi(b)\uz=b\o S({}b\th)\ot \phi(b\t)$,
        \item[(2)] $\phi(a\umo)\ot a\uz= S^{-1}({}a\th)a\o\ot a\t$,
        \item[(3)] $\phi(b\ra a)=S^{-1}(a\t)\phi(b)a\o$,
        \item[(4)] $b'\ra \phi(b)=b\o b' S(b\t)$,
    \end{itemize}
    for all $a\in A$, $b, b'\in B$. Then $A\rlbicross B$ is an (anti-)right Hopf 2-algebra, with the source and target map
    \[s(b)=1\ot b,\qquad t(b)=\phi(b\o)\ot b\t,\]
    the coproduct and counit of the bialgebroid
    \[\blacktriangle(a\ot b)=a\o\ot 1\ot_{B} a\t\ot b,\qquad \varepsilon_{A \rlbicross B}(a\ot b)=\varepsilon(a)b,\]
    for any $a\in A$ and $b\in B$. The inverse of the canonical map $\lambda$ is determined by
    \[(a\ot b)_{-}\ot_{B^{op}}(a\ot b)_{+}=S(a\o)\ot 1\ot_{B^{op}}a\t\ot b,\]
    and the inverse of the anti-canonical map $\mu$ is determined by
    \[(a\ot b)_{[+]}\ot^{B^{op}}(a\ot b)_{[-]}=a\o\ot 1\ot^{B^{op}}\phi(b\o)S^{-1}({}a\th)\ot b\t\ra S^{-1}(a\t).\]
    We call such a bicrossproduct Hopf algebra $(A\rlbicross B, \phi)$  a \textup{bicrossed module} of Hopf algebras.
    Moreover, if the antipode of $A$ satisfies $S^{2}=\id_{A}$, then $A\rlbicross B$ is a Hopf 2-algebra with the antipode of the full Hopf algebroid
    \[S_{A\rlbicross B}(a\ot b)=\phi(b\o)S(a\t)\ot b\t\ra S(a\o).\]

\end{thm}

\begin{proof}
    The source map is clearly a bialgebra map. We can also see the target map $t:B^{op}\to A\rlbicross B$ is a bialgebra map,
    \begin{align*}
        t(b)t(b')=&\phi(b\o)\phi(b'\o)\o\ot b\t \ra\phi(b'\o)\t b'\t\\
        =&\phi(b\o)\phi(b'\o)\ot b\t \ra\phi(b'\t) b'\th\\
        =&\phi(b\o)\phi(b'\o)\ot b'\t b\t\\
        =&\phi(b'\o b\o)\ot b'\t b\t\\
        =&t(b'b),
    \end{align*}
    where the 2nd step uses the fact that $\phi$ is an coalgebra map, the 3rd step uses condition (4), the 4th step uses the fact that $\phi$ is an antialgebra map from  $B$ to $A$. We also have
    \begin{align*}
        \Delta(t(b))=&\phi(b\o)\ot \phi(b\t)\umo b\th \ot \phi(b\t)\uz\ot{}b\fo\\
        =&\phi(b\o)\ot b\t\ot \phi({}b\th)\ot b\fo\\
        =&t(b\o)\ot t(b\t),
    \end{align*}
    where the 2nd step uses condition (1).
    By the same method, we can also see the source and target maps have commuting ranges,
    \begin{align*}
        s(b')t(b)=\phi(b\o)\ot (b'\ra \phi(b\t)){}b\th=\phi(b\o)\ot b\t b'=t(b)s(b').
    \end{align*}
We can see both of the counit and coproduct $\blacktriangle$ are $B$-bilinear, here we only check the coproduct is left $B$-linear:
\begin{align*}
    \blacktriangle((a\ot b)t(b'))=&\blacktriangle(a\phi(b'\o)\ot b'\t b)\\
    =&(a\o \phi(b'\o)\ot 1)\ot_{B}(a\t\ot b)(\phi(b'\t)\ot b'\th)\\
    =&(a\o \phi(b'\o)\ot 1)\ot_{B}(a\t\ot b)t(b'\t)\\
    =&(a\o \phi(b'\o)\ot 1)s(b'\t)\ot_{B}(a\t\ot b)\\
    =&(a\o \ot 1)t(b')\ot_{B}(a\t\ot b).
\end{align*}
So it is not hard to see $A\rlbicross B$ is a $B$-coring.
Now, let's show the image of $\blacktriangle$ belongs to the Takeuchi product,
\begin{align*}
    s(b')(a\o\ot 1)&\ot_{B} (a\t\ot b)\\
    =&a\o\ot b'\ra a\t\ot_{B} {}a\th\ot b\\
    =&(a\o\ot 1)s(b'\ra a\t)\ot_{B}(a\th\ot b)\\
    =&(a\o\ot 1)\ot_{B}(a\th\ot b)t(b'\ra a\t)\\
    =&(a\o\ot 1)\ot_{B}a\th \phi((b'\ra a\t)\o)\ot b\ra \phi((b'\ra a\t)\t)(b'\ra a\t)\th\\
   =&(a\o\ot 1)\ot_{B}a\th \phi((b'\ra a\t)\o)\ot (b'\ra a\t)\t b\\
   =&(a\o\ot 1)\ot_{B}a\fo \phi((b'\o\ra a\t)a\th\umo)\ot b'\t\ra a\th\uz b\\
    =&(a\o\ot 1)\ot_{B}a\fo S^{-1}(a\th\th) a\th\o S^{-1}(a\t\t)\phi(b'\o)a\t\o\ot b'\t\ra a\th\t b \\
    =&(a\o\ot 1)\ot_{B}\phi(b'\o)a\t\ot b'\t\ra a\th b\\
    =&(a\o\ot 1)\ot_{B} t(b')(a\t\ot b),
\end{align*}
where the 6th step uses condition (iii) of bicrossproduct, and the 7th step uses conditions (2) and (3) with the fact that $\phi$ is an antialgebra map from $B$ to $A$. It is not hard to see the coproduct $\blacktriangle$ is an algebra map. It is also a straightforward computation to check the counit satisfy $\varepsilon_{A\rlbicross B}(XY)=\varepsilon_{A\rlbicross B}(s(\varepsilon_{A\rlbicross B}(X))Y)=\varepsilon_{A\rlbicross B}(t(\varepsilon_{A\rlbicross B}(X))Y)$ for any $X, Y\in A\rlbicross B$. Therefore, $A\rlbicross B$ is a right bialgebroid over $B$. Moreover, we can see $\varepsilon_{A\rlbicross B}$ is a coalgebra map. Now, let's check condition (iv) of Definition \ref{def. Hopf-2-algebra}. On the one hand,
\begin{align*}
    (\blacktriangle\otimes \blacktriangle)\circ \Delta(a\ot b)=(a\o\ot 1)\ot_{B}(a\t\ot a\th\umo b\o)\ot (a\th\uz\o\ot 1)\ot_{B}(a\th\uz\t\ot b\t).
\end{align*}
On the other hand,
\begin{align*}
    (\id\otimes &\mathrm{flip}\otimes \id)\circ (\Delta\otimes_{B}\Delta)\circ \blacktriangle(a\ot b)\\
    =&(a\o\ot a\t\umo)\ot_{B}(a\th\ot a\fo\umo b\o)\ot (a\t\uz\ot 1)\ot_{B}(a\fo\uz\ot b\t)\\
    =&(a\o\ot 1) s(a\t\umo)\ot_{B}(a\th\ot a\fo\umo b\o)\ot (a\t\uz\ot 1)\ot_{B}(a\fo\uz\ot b\t)\\
    =&(a\o\ot 1) \ot_{B}(a\th\ot a\fo\umo b\o)t(a\t\umo)\ot (a\t\uz\ot 1)\ot_{B}(a\fo\uz\ot b\t)\\
    =&(a\o\ot 1) \ot_{B}(a\th \phi(a\t\umo\o)\ot a\t\umo\t a\fo\umo b\o)\ot (a\t\uz\ot 1)\ot_{B}(a\fo\uz\ot b\t)\\
    =&(a\o\ot 1) \ot_{B}(a\th \phi(a\t\umo)\ot a\t\uz\umo a\fo\umo b\o)\ot (a\t\uz\uz\ot 1)\ot_{B}(a\fo\uz\ot b\t)\\
    =&(a\o\ot 1) \ot_{B}(a\th S^{-1}(a\t\th) a\t\o\ot a\t\t\umo a\fo\umo b\o)\ot (a\t\t\uz\ot 1)\ot_{B}(a\fo\uz\ot b\t)\\
    =&(a\o\ot 1) \ot_{B}(a\t\ot a\th\umo a\fo\umo b\o)\ot (a\th\uz\ot 1)\ot_{B}(a\fo\uz\ot b\t)\\
    =&(a\o\ot 1) \ot_{B}(a\t\ot a\th\umo b\o)\ot (a\th\uz\o\ot 1)\ot_{B}(a\th\uz\t\ot b\t),
\end{align*}
where the 6th step uses condition (2). Now we first show the canonical map $\lambda$ is invertible,
\begin{align*}
    \lambda\circ \lambda^{-1}(1\ot 1\ot_{B}a\ot b)=(S(a\o)\ot 1)(a\t\ot 1)\ot_{B}(a\th\ot b)=1\ot 1\ot_{B}a\ot b;
\end{align*}
and
\begin{align*}
    \lambda^{-1}\circ\lambda(1\ot 1\ot_{B^{op}}a\ot b)=(a\o\ot 1)(S(a\t)\ot 1)\ot_{B^{op}}(a\th\ot b)=1\ot 1\ot_{B^{op}}a\ot b.
\end{align*}
Second, we show the anti-canonical map $\mu$ is invertible,
\begin{align*}
    \mu\circ\mu^{-1}(a\ot b\ot_{B} 1\ot 1)=&a\o\ot 1\ot_{B} \phi(b\o) S^{-1}(a\five)a\t\ot b\t\ra (S^{-1}(a\fo)a\th)\\
    =&a\ot 1\ot_{B} \phi(b\o)\ot b\t\\
    =&a\ot b\ot_{B}1\ot 1;
\end{align*}
and
\begin{align*}
    \mu^{-1}\circ \mu (a\ot b\ot^{B^{op}} 1\ot 1)=&a\o\ot 1\ot^{B^{op}} (a\th\ot b)(S^{-1}(a\t)\ot 1)\\
    =&a\o\ot 1\ot^{B^{op}}1\ot b\ra S^{-1}(a\t)\\
    =&s(b\ra S^{-1}(a\t))(a\o\ot 1)\ot^{B^{op}}1\ot 1\\
    =&a\ot b\ot^{B^{op}}1\ot 1.
\end{align*}

Moreover, if the antipode of $A$ satisfies $S^{2}=\id_{A}$, then we can see $S_{A\rlbicross B}$ is invertible with $S_{A\rlbicross B}=S_{A\rlbicross B}^{-1}$,
\begin{align*}
    S_{A\rlbicross B}^{2}&(a\ot b)\\
    =&\phi((b\t\ra S(a\o))\o) S(\phi(b\o)\t S(a\t)\t)\ot (b\t\ra S(a\o))\t\ra S(\phi(b\o)\o S(a\t)\o)\\
    =&\phi((b\th\ra S(a\t))S(a\o)\umo)a\th S(\phi(b\t))\ot(b\fo\ra S(a\o)\uz)\ra (a\fo S(\phi(b\o)))\\
    =&\phi(S(a\o)\umo)\phi(b\th\ra S(a\t))a\th S(\phi(b\t))\ot(b\fo\ra S(a\o)\uz)\ra (a\fo S(\phi(b\o)))\\
    =&a\o S(a\th) a\fo \phi(b\th)S(a\five) a\si S(\phi(b\t))\ot b\fo\ra(S(a\t)a\se S(\phi(b\o)))\\
    =&a\phi(b\th)S(\phi(b\t))\ot b\fo\ra S(\phi(b\o))\\
    =&a\ot b\t\ra S(\phi(b\o))\\
    =&a\ot S(b\t)b\th b\o\\
    =&a\ot b,
\end{align*}
where the 6th step uses the fact that $S^{-1}(\phi(b))=S^{-1}(\phi(b\th))\phi(b\t) \phi(S(b\o))=\phi(S(b))$.
To show the antipode $S_{A\rlbicross B}$ is an antialgebra map, we can check on the one hand,
\begin{align*}
    S_{A\rlbicross B}&((a\ot b)(a'\ot b'))\\
    =&\phi((b\ra a'\th)\o b'\o)S(a\t a'\t)\ot ((b\ra a'\th)\t b'\t)\ra S(a\o a'\o)\\
    =&\phi(b'\o) \phi((b\o\ra a'\th)a'\fo\umo)S(a'\t)S(a\t)\ot ((b\t\ra a'\fo\uz)b'\t)\ra S(a\o a'\o)\\
    =&\phi(b'\o) S(a'\fo\th)a'\fo\o S(a'\th\t)\phi(b\o)a'\th\o S(a'\t) S(a\t)\\
    {}&\ot ((b\t\ra a'\fo\t)b'\t)\ra S(a\o a'\o)\\
    =&\phi(b'\o)S(a'\t)\phi(b\o)S(a\t)\ot (b\t( b'\t\ra S(a'\o)))\ra S(a\o).
\end{align*}
 On the other hand,
 \begin{align*}
     S_{A\rlbicross B}&(a'\ot b') S_{A\rlbicross B}(a\ot b)\\
     =&\phi(b'\o)S(a'\t)\phi(b\o)S(a\th)\ot (b'\t\ra (S(a'\o) \phi(b\t)S(a\t)))(b\th \ra S(a\o))\\
     =&\phi(b'\o)S(a'\t)\phi(b\o)S(a\t)\ot (b\t( b'\t\ra S(a'\o)))\ra S(a\o).
 \end{align*}
We can also see that
\begin{align*}
    S^{-1}_{A\rlbicross B}(s(b))=S_{A\rlbicross B}(s(b))=\phi(b\o)\ot b\t=t(b).
\end{align*}
For any $X=a\ot b\in A\rlbicross B$, we have on the other hand,
\begin{align*}
   S_{A\rlbicross B}(X\htw)\ho &\ot_{B}X\ho
 S_{A\rlbicross B}(X\htw)\htw\\
 =&\phi(b\o)S(a\fo)\ot 1\ot_{B}(a\o\ot 1)(\phi(b\t)S(a\th)\ot b\th \ra S(a\t))\\
 =&\phi(b\o)S(a\fo)\ot 1\ot_{B}a\o\phi(b\t)S(a\th)\ot b\th \ra S(a\t),
\end{align*}
on the other hand,
\begin{align*}
    S_{A\rlbicross B}(X)\ot_{B}1
=&(\phi(b\o)S(a\t)\ot 1)s(b\t\ra S(a\o))\ot_{B}1\ot 1\\
=&(\phi(b\o)S(a\t)\ot 1)\ot_{B}(1\ot 1)t(b\t\ra S(a\o))\\
=&(\phi(b\o)S(a\t)\ot 1)\ot_{B}\phi((b\t\ra S(a\o))\o)\ot (b\t\ra S(a\o))\t\\
=&(\phi(b\o)S(a\th)\ot 1)\ot_{B}\phi(S(a\o)\umo)\phi(b\t\ra S(a\t))\ot b\th \ra S(a\o)\uz\\
=&(\phi(b\o)S(a\si)\ot 1)\ot_{B} a\o S(a\th) a\fo \phi(b\t) S(a\five)\ot b\th\ra S(a\t)\\
=&\phi(b\o)S(a\fo)\ot 1\ot_{B}a\o\phi(b\t)S(a\th)\ot b\th \ra S(a\t).
\end{align*}

And
\begin{align*}
    X\htw S^{-1}_{A\rlbicross B}(X\ho)\ho&\ot_{B} S^{-1}_{A\rlbicross B}(X\ho)\htw\\
    =&a\th S(a\t)\o\ot b\ra S(a\t)\t\ot_{B} S(a\o)\ot 1\\
    =&1\ot b\ra S(a\t)\ot_{B} S(a\o)\ot 1\\
    =&1\ot 1\ot_{B}(S(a\o)\ot 1) t(b\ra S(a\t))\\
    =&1\ot 1\ot_{B}S(a\o)\phi(S(a\t)\umo)\phi(b\o\ra S(a\th))\ot b\t\ra S(a\t)\uz\\
    =&1\ot 1\ot_{B}\phi(b\o)S(a\t)\ot b\t\ra S(a\o)\\
    =&1\ot 1\ot_{B} S^{-1}_{A\rlbicross B}(X).
\end{align*}
\end{proof}

\begin{rem}
    It is given in Corollary 5.9 of \cite{XH} that given two Hopf algebras $A$ and $B$(with $B$ commutative), if $A$ is a left $B$-comodule coalgebra such that a version of (1) and (2) in Theorem \ref{thm. bicrossed module} is satisfied, then $(H=A\ot B, B)$ is a Hopf 2-algebra with factorwise multiplication. However, in that case, we have to require that $B$ is commutative.
\end{rem}

It is known (\cite{Majid2}) that given a Hopf algebra with a bijective antipode, $H\rlbicross H_{cop}$ is a bicrossproduct. Here $H_{cop}$ is a right module algebra of $H$ with the action given by $g\ra h=S(h\o)g h\t$, and $H$ is a left comodule coalgebra of $H_{cop}$ with  the coaction given by $\delta(h)=S(h\o)h\th\ot h\t$ for any $h\in H$ and $g\in H_{cop}$.

\begin{thm}
Let $H$ be a Hopf algebra with bijective antipode $S$, then $(H\rlbicross H_{cop}, S^{-1})$ is a bicrossed module of Hopf algebras.
\end{thm}
\begin{proof}
    We know $S^{-1}:H_{cop}^{op}\to H$ is a Hopf algebra map. For condition (1), we can see
    \begin{align*}
        S^{-1}(g)\umo\ot S^{-1}(g)\uz=&S(S^{-1}(g)\o)S^{-1}(g)\th\ot S^{-1}(g)\t=g\th S^{-1}(g\o)\ot S^{-1}(g\t),
    \end{align*}
since $S^{-1}$ is the antipode of $H_{cop}$, we have (1). For condition (2),
\begin{align*}
    S^{-1}(h\umo)\ot h\uz=S^{-1}(S(h\o)h\th)\ot h\t=S^{-1}(h\th)h\o\ot h\t.
\end{align*}
For condition (3),
\begin{align*}
    S^{-1}(g\ra h)=S^{-1}(S(h\o)g h\t)=S^{-1}(h\t) S^{-1}(g) h\o.
\end{align*}
For condition (4),
\begin{align*}
    g'\ra S^{-1}(g)=S(S^{-1}(g)\o)g' S^{-1}(g)\t=g\t g' S^{-1}(g\o),
\end{align*}
where we use the fact that $S^{-1}$ is the antipode of $H_{cop}$.
\end{proof}

\subsection{Example} Recall \cite{Swe} \cite{Taft} that a Sweedler's Hopf algebra $H^{4}$ is a Hopf algebra generated by two elements $x, g$, with the commutation relation:
\[x^{2}=0,\quad g^{2}=1, \quad xg+gx=0,\]
the coproduct
\[\Delta(x)=1\ot x+x\ot g,\quad \Delta(g)=g\ot g,\]
the counit
\[\varepsilon(x)=0, \quad\varepsilon(g)=1,\]
the antipode
\[S(x)=-xg^{-1},\quad S(g)=g^{-1}.\]
It is not hard to see that $S^{2}=\id$, so by Theorem \ref{thm. bicrossed module}, $H^{4}\rlbicross H^{4}_{cop}$ is a Hopf 2-algebra. We also compute the Hopf algebroid antipode on the generators. More precisely, we have
\begin{align*}
    S_{H^{4}\rlbicross H^{4}_{cop}}(g\ot g)=&1\ot g,\\
    S_{H^{4}\rlbicross H^{4}_{cop}}(g\ot x)=&-1\ot x-x\ot 1,\\
    S_{H^{4}\rlbicross H^{4}_{cop}}(x\ot g)=&x\ot g+1\ot 2gx,\\
    S_{H^{4}\rlbicross H^{4}_{cop}}(x\ot x)=&(-x^{2}-x)\ot 1+x\ot x.
\end{align*}

\end{document}